\documentclass[12pt]{amsart}

\usepackage{amsthm}
\usepackage[dvips]{color}
\usepackage{url}
\usepackage{subfigure}
\usepackage{graphicx}
\usepackage[dvips]{color}
\usepackage{placeins}

\newtheorem{theorem}{Theorem}[section]
\newtheorem{lemma}[theorem]{Lemma}

\theoremstyle{definition}

\theoremstyle{remark}
\newtheorem{rem}[theorem]{Remark}

\numberwithin{equation}{section}

\begin{document}
\sloppy
\title[Boundary Conditions for Fractional Diffusion]{Boundary Conditions for Fractional Diffusion}

\author[Baeumer]{Boris Baeumer}
\address{Boris Baeumer, University of Otago, New Zealand}
\email{bbaeumer@maths.otago.ac.nz}
\thanks{Baeumer was partially supported by the Marsden Fund administered by the Royal Society of New Zealand.}

\author[Kov\'acs]{Mih\'aly Kov\'acs}
\address{Mih\'aly Kov\'acs, Chalmers University of Technology, Sweden}
\email{mihaly@chalmers.se}
\thanks{Kov\'acs was partially supported by the Marsden Fund administered by the Royal Society of New Zealand.}

\author[Meerschaert]{Mark M. Meerschaert}
\address{Mark M. Meerschaert, Department of Statistics and Probability, Michigan State University}
\email{mcubed@stt.msu.edu}
\thanks{Meerschaert was partially supported by ARO MURI grant W911NF-15-1-0562 and NSF grants DMS-1462156 and EAR-1344280.}

\author[Sankaranarayanan]{Harish Sankaranarayanan}
\address{Harish Sankaranarayanan, Michigan State University, USA}
\email{harish@msu.edu}
\thanks{Sankaranarayanan was supported by ARO MURI grant W911NF-15-1-0562.}

\date{\today}
%

\begin{abstract}
This paper derives physically meaningful boundary conditions for fractional diffusion equations, using a mass balance approach.  Numerical solutions are presented, and theoretical properties are reviewed, including well-posedness and steady state solutions.  Absorbing and reflecting boundary conditions are considered, and illustrated through several examples.  Reflecting boundary conditions involve fractional derivatives.  The Caputo fractional derivative is shown to be unsuitable for modeling fractional diffusion, since the resulting boundary value problem is not positivity preserving.
\end{abstract}

\maketitle

\section{Introduction}
The space-fractional diffusion equation replaces the second derivative or Laplacian in the traditional diffusion equation with a fractional derivative.  Fractional derivatives were invented soon after their integer-order counterparts, and by now have become an established field of study with a wide variety of applications in science and technology \cite{Herrmann,KRS,MainardiWaves,mainardi1997fractals,Metzler2000,Metzler2004a,Podlubny,Samko}.  Many effective numerical methods have been developed for fractional differential equations, along with proofs of stability and consistency \cite{Deng, DFF, FixRoop,  LAT, Lynch, frade, 2sided, TMSADIIE, OM, Podlubny, CNfde,Yuste, ZLPM}.  In many cases, the underlying theory is still being developed, and indeed it is not known whether the problems are well-posed, with unique solutions.  Part of the difficulty has been that fractional derivatives are nonlocal operators, and hence the concept of a boundary condition takes on new meaning \cite{SpaceTimeFrac,FCAAnonlocal}.

This paper considers space-fractional diffusion equations on the unit interval $0\leq x\leq 1$ with absorbing or reflecting boundary conditions.  Both Riemann-Liouville and Caputo flux forms are considered, and the profound difference in the solutions is illustrated. To specify a fractional diffusion equation on a bounded domain, appropriate boundary conditions must be enforced.  We discuss absorbing (Dirichlet) and reflecting (Neumann) boundary conditions, which can take a very different form for a fractional evolution equation.  For example, we will see that the appropriate Neumann boundary condition sets a fractional derivative equal to zero at the boundary, not the first derivative as in the traditional diffusion equation.  We also show that the Caputo form does not preserve positivity, and hence cannot provide a suitable model for anomalous diffusion.

By varying the type of space-fractional derivative and the boundary conditions, we obtain a number of possible fractional diffusion equations on the unit interval.  For each of these, we develop and apply a suitable numerical solution method.  We also review the underlying theory from the point of view of abstract evolution equations, semigroups and generators.  Well-posedness is verified, including uniqueness of solutions, and steady-state solutions are identified.

\section{Fractional boundary value problems}\label{Sec2}

Consider the fractional diffusion equation
\begin{equation}
\label{RL-CauchyEq}
\frac{\partial}{\partial t}u(x,t)= C\, {\mathbb D}^\alpha u(x,t)
\end{equation}
on the entire real line, where the Riemann-Liouville fractional derivative
\begin{equation}
\label{RLdef}
{\mathbb D}^\alpha  u(x,t) = \frac{1}{\Gamma(n-\alpha)} \frac{\partial^n}{\partial x^n} \int_{-\infty}^x u(y,t)(x-y)^{n-\alpha-1} dy
\end{equation}
for $\alpha>0$ and $n-1<\alpha\leq n$.  Note that \eqref{RLdef} is a nonlocal operator that depends on the values of $u(y,t)$ at every point $y<x$.  The exact analytical solution to \eqref{RL-CauchyEq} can be written in terms of a stable probability density function.  Although this analytical solution cannot be computed in closed form, there are readily available codes that compute the stable density, and these can be used to plot the solutions to \eqref{RL-CauchyEq}.  See for example \cite[Chapter 5]{FCbook}.

However, if we restrict the fractional diffusion to a finite interval, then there are no known analytical solutions, and numerical methods must be used.  First consider the fractional diffusion equation
\begin{equation}
\label{DcDDeq}
\frac{\partial}{\partial t}u(x,t)= C\, {\mathbb D}^\alpha_{[0,x]} u(x,t)
\end{equation}
on the state space $0< x< 1$ with initial condition $u(x,0)=u_0(x)$.  On this finite domain, we define the Riemann-Liouville fractional derivative
\begin{equation}
\label{RLdef0}
{\mathbb D}^\alpha_{[0,x]}  u(x,t) = \frac{1}{\Gamma(n-\alpha)} \frac{\partial^n}{\partial x^n} \int_0^x u(y,t)(x-y)^{n-\alpha-1} dy ,
\end{equation}
the only difference from \eqref{RLdef} being the lower limit of integration.  This is still a nonlocal operator, since it depends on the values of $u(y,t)$ at every point $0<y<x$.

\section{Absorbing boundary conditions}\label{SecDD}
Now let us impose a zero boundary condition at each endpoint:
\begin{equation}
\label{DcDDbc}
u(0,t)= u(1,t)=0\quad\text{for all $t\geq 0.$}
\end{equation}
The zero Dirichlet boundary conditions \eqref{DcDDbc} are usually called absorbing boundary conditions, but do they have the same meaning for a nonlocal operator?  To illuminate this issue, let us develop a numerical method to solve the fractional diffusion equation \eqref{DcDDeq}, paying special attention to meaning of the zero boundary conditions \eqref{DcDDbc}.

The fractional derivative \eqref{RLdef} can be approximated using the Gr\"unwald-Letnikov formula \cite[Proposition 2.1]{FCbook}
\begin{equation}
\label{GrunwaldLetnikov}
{\mathbb D}^\alpha  u(x,t) =\lim_{h\to 0}h^{-\alpha}\sum_{i=0}^\infty (-1)^i \begin{pmatrix} \alpha\\i\end{pmatrix} f(x-ih)
\end{equation}
where the Gr\"unwald weights are given by
\begin{equation}
\label{Grunwald}
g^\alpha_i= (-1)^i \begin{pmatrix} \alpha\\i\end{pmatrix}= \frac{(-1)^i\Gamma(\alpha+1)}{\Gamma(i+1)\Gamma(\alpha-i+1)}
\end{equation}
for all $i\geq 0$.  Since the finite domain fractional derivative \eqref{RLdef0} is equivalent to the Riemann-Liouville fractional derivative of a function that vanishes for $x<0$, we immediately obtain that
\begin{equation}
\label{Grunwald0}
{\mathbb D}^\alpha_{[0,x]}  u(x,t) =\lim_{h\to 0}h^{-\alpha}\sum_{i=0}^{[x/h]} g^\alpha_i f(x-ih).
\end{equation}
This approximation can be used to construct numerical solutions to the fractional diffusion equation, but the resulting methods are unstable \cite[Proposition 2.3]{frade}.  Instead, we apply a shifted Gr\"unwald formula
\begin{equation}
\label{Grunwald0s}
{\mathbb D}^\alpha_{[0,x]}  u(x,t) \approx h^{-\alpha}\sum_{i=0}^{[x/h]+1} g^\alpha_i f(x-(i-1)h)
\end{equation}
which results in a stable method \cite[Theorem 2.7]{frade}.

To illuminate the role of the boundary conditions, first consider the fractional diffusion equation \eqref{RL-CauchyEq} on the real line.  As a thought experiment, discretize $x_j=jh$ and $t_k=k\Delta t$ and apply the Gr\"unwald approximation to obtain the explicit Euler scheme
\begin{equation}
\label{Euler}
u(x_j,t_{k+1})=u(x_j,t_k)+C h^{-\alpha}\sum_{i=0}^{\infty} g^\alpha_i u(x_{j-i+1},t_k) \Delta t.
\end{equation}
The Gr\"unwald weights are $g^\alpha_0=1$, $g^\alpha_1=-\alpha$, $g^\alpha_2=\alpha(\alpha-1)/2!$ and so forth, and note that $g^\alpha_i>0$ for all $i\neq 1$.
The scheme is mass-preserving because \cite[Eq.\ (2.11)]{FCbook}
\begin{equation}
\label{sum0}
\sum_{i=0}^\infty g^\alpha_i=0 ,
\end{equation}
and hence to understand a physical model of the fractional diffusion, it will suffice to consider $u(x_j,t_k)h$ as the mass at location $x_j$ at time $t_k$.  The total mass $M_k=\sum_j u(x_j,t_k)h$ does not vary with time $t_k$, but rather remains equal to the initial mass $M_0=\sum_j u_0(x_j)h$.   The scheme moves a mass $C \Delta t h^{-\alpha-1}g^\alpha_i\ u(x_{j-i+1},t_k) h$ from location $x_{j-i+1}$ to location $x_j$ when $i\neq 1$.  The total mass $C \Delta t h^{-\alpha}\alpha\, u(x_{j},t_k)$ moved out of location $x_j$ is equal to the sum of the amounts moved from location $x_j$ to another location, because $\sum_{i\neq 1} g^\alpha_i=\alpha$.   In this scheme, mass can be transported large distances to the right, but only one step size $h$ to the left.  Note that the scheme \eqref{Euler} is also positivity preserving for $C \alpha h^{-\alpha}\Delta t\leq 1$, since a fraction $\leq 100\%$ of the mass at each point is removed, and then redistributed.

Now we want to restrict to the unit interval $0\leq x\leq 1$ and impose the zero boundary conditions \eqref{DcDDbc}.  Since we are solving a nonlocal problem, this requires some care.  Unlike a traditional diffusion equation, the Euler scheme \eqref{Euler} moves mass a long distance in one time step, for any step size. That mass can land outside the unit interval, and then it must be accounted for in the scheme.  Part of the picture is to understand how the Gr\"unwald approximation \eqref{Grunwald0} accounts for this mass.  The remaining part is to understand the zero boundary conditions.

Let us note that the discretization of the fractional diffusion equation \eqref{DcDDeq} on the bounded domain using \eqref{Grunwald0s} takes the form
\begin{equation}
\label{Euler0}
u(x_j,t_{k+1})=u(x_j,t_k)+C h^{-\alpha}\sum_{i=0}^{j+1} g^\alpha_i u(x_{j-i+1},t_k) \Delta t ,\quad \forall\ 0\leq j\leq n.
\end{equation}
Comparing with \eqref{Euler}, we can see that no mass is moved to location $x_j$ from any location $x_{j-i+1}$ when $i>j+1$, i.e., when $x_{j-i+1}<0$ lies outside the domain $0\leq x\leq 1$.


Now we impose the boundary conditions \eqref{DcDDbc} by setting $u(x_j,t_k)=0$ when $j=0$ (location $x_j=0$) or $j=n$ (location $x_j=1$), where $nh=1$.  Since our initial condition $u_0(x)$ must also satisfy the boundary conditions, we start with all the mass inside the open interval $0<x<1$.  To enforce the boundary conditions, we have to modify the Euler scheme \eqref{Euler0}.  After a simple change of variables, we can write \eqref{Euler0} in the form
\begin{equation}
\label{Euler00}
u(x_j,t_{k+1})=u(x_j,t_k)+C h^{-\alpha}\sum_{i=0}^{n} b_{ij} u(x_{i},t_k) \Delta t ,\quad \forall\ 0\leq j\leq n,
\end{equation}
where $b_{ij}=g^\alpha_{j-i+1}$ for $i\leq j+1$ and $b_{ij}=0$ for $i>j+1$.   Next we will modify certain coefficients $b_{ij}$ to enforce the boundary conditions.  First consider the left end point $x_0=0$.  Since the mass at this location has to remain zero,
\[0=\sum_{i=0}^{1} b_{i0} u(x_{i},t_k) .\]
Since $u(x_{0},t_k)=0$ for all $k$, this requires $b_{10}=0$.  Now the mass $C \Delta t h^{-\alpha-1}g^\alpha_0 u(x_{1},t_k)h $ that would have been transported from location $x_1$ to location $x_0$ is instead removed from the system, to enforce the zero boundary condition.
Next consider the right end point $x_n=1$.  Since the mass at this location has to remain zero, we require
\[0=\sum_{i=0}^{n} b_{in} u(x_{i},t_k).\]
Since $u(x_{n},t_k)=0$ for all $k$, and since all $u(x_{i},t_k)\geq 0$ for step size $\Delta t \leq h^{\alpha}/C \alpha$ and a nonnegative initial condition, we must have $b_{in}=0$ for all $i=0,1,2,n-1$.   This change alters \eqref{Euler0} by taking the mass $C \Delta t h^{-\alpha}g^\alpha_{n-i+1} u(x_i,t_k)$ that would have been transported from location $x_i<1$ to location $x_n=1$ and removing it from the system. The resulting scheme can be written in the form \eqref{Euler00} where
\begin{equation}
\label{EulerDDweights}
b_{ij}=\begin{cases}
g^\alpha_{j-i+1} &\text{if $0<j<n$ and $i\leq j+1$,}\\
0&\text{otherwise.}
\end{cases}
\end{equation}
To interpret \eqref{EulerDDweights}, recall that $C h^{-\alpha}b_{ij} u(x_{i},t_k) \Delta t $ is the mass transferred from location $x_i$ to location $x_j$ during this time step.

\begin{rem}
The astute reader will notice that \eqref{Euler0} with $j=n$ involves the mass at location $x_{n+1}=1+h$ when $j=n$, and this $x_{n+1}$ term does not appear in \eqref{Euler00}.  We could indeed track the mass moved to the location $x_{n+1}$, which is outside the domain, but with the zero boundary condition $u(x_n,t_k)=0$, none of this mass can ever come back into the domain.  Indeed, mass from location $x_{n+1}$ can only move left one step to location $x_n=1$, and the zero boundary condition forbids this.  In other words, if we did include state $x_{n+1}$ in our scheme, then we would also conclude $b_{n+1,n}=0$ by the same argument that $b_{in}=0$ for $0\leq i\leq n-1$.  Hence we need not track the mass at this location.
\end{rem}


In summary, the fractional diffusion equation \eqref{DcDDeq} on $0\leq x\leq 1$ with zero boundary conditions \eqref{DcDDbc} is indeed a model with absorbing boundary conditions.  As compared to the Euler scheme on the entire real line, here the mass scheduled for transport beyond the boundary of the unit interval is instead deleted from the system, or absorbed.  This scheme is also positivity preserving so long as $C \alpha\Delta t h^{-\alpha}<1$, since a fraction of the mass at each point is removed, and then redistributed or absorbed.

Write $\beta=C h^{-\alpha}\Delta t$, $u_j^k=u(x_j,t_k)$, the solution vector ${\mathbf u}_k=[u_0^k,\ldots,u_n^k]$, and the $(n+1) \times (n+1)$ iteration matrix $B=[b_{ij}]$.  Then we can express the explicit Euler scheme \eqref{Euler00} in vector-matrix form
\begin{equation}
\label{Euler0e}
{\mathbf u}_{k+1}={\mathbf u}_k + \beta  {\mathbf u}_k B .
\end{equation}
In this form, the $ij$ entry of the matrix $B$ is proportional to the rate at which mass is transferred from location $x_i$ to location $x_j$.  Equivalently, we can write
\begin{equation}
\label{Euler0eT}
{\mathbf u}_{k+1}^T={\mathbf u}_k^T + \beta B^T {\mathbf u}_k^T .
\end{equation}

The formulation \eqref{Euler0eT} is traditional in numerical analysis, e.g., see \cite[p.\ 4]{2sided}, while \eqref{Euler0e} is used for Markov chains, e.g., see \cite[Section 8.1]{MathModeling}.

The explicit Euler scheme \eqref{Euler0eT} is stable under a step size condition $\alpha\beta<1$, or equivalently, $\Delta t<h^{\alpha}/C\alpha $, see \cite[Proposition 2.1]{2sided}.  The implicit Euler scheme
\begin{equation}
\label{Euler0i}
{\mathbf u}_{k+1}^T={\mathbf u}_k^T + \beta B^T {\mathbf u}_{k+1}^T
\end{equation}
is unconditionally stable \cite[Theorem 2.7]{frade}.  As noted in the Introduction, by now there are a wide variety of numerical methods to solve this problem.  For example, the explicit Euler scheme \eqref{Euler0e} can be viewed as the temporal discretization of a linear system of ordinary differential equations (method of lines, e.g., see \cite{BKStams,LAT}), and then any standard method for solving the linear system can be employed.

\begin{rem}\label{theoryDDrem}
Theoretical properties of the solution are discussed in \cite{SpaceTimeFrac,Sankaranarayanan2014}.  There it is shown that the Cauchy problem \eqref{DcDDeq} on $0\leq x\leq 1$ with zero boundary conditions \eqref{DcDDbc} (or equivalently, zero exterior condition) is well-posed: There exists a unique solution for any initial condition $u_0(x)$ that depends continuously on this initial function.  The general theory in \cite{SpaceTimeFrac} applies on the Banach space $C_0(0,1)$ of continuous functions that vanish at the end points, with the supremum norm.  In \cite{Sankaranarayanan2014} the Banach space $L^1[0,1]$ is considered.  Since both the implicit and explicit Euler methods are consistent, and stable (in the explicit case, under a step size condition on $\Delta t$), and since the problem \eqref{DcDDeq} on $0\leq x\leq 1$ with zero boundary conditions is well-posed, the Lax Equivalence Theorem \cite[p.\ 45]{Richtmyer} implies that either of these Euler methods will converge to the unique solution as $h\to 0$ and $\Delta t\to 0$.  The same is true for any other stable, consistent numerical method.  The theory in \cite{SpaceTimeFrac} also relates the Cauchy problem \eqref{DcDDeq} on $0\leq x\leq 1$ with zero boundary conditions to a probability model, which implies that the problem is positivity preserving.  The analysis in \cite{Sankaranarayanan2014} also computes the exact domain of the generator ${\mathbb D}^\alpha_{[0,x]}$ on $L^1[0,1]$ with zero boundary conditions.
\end{rem}

\begin{figure}
\begin{center}
\hskip-0.5in
\includegraphics[width=4in]{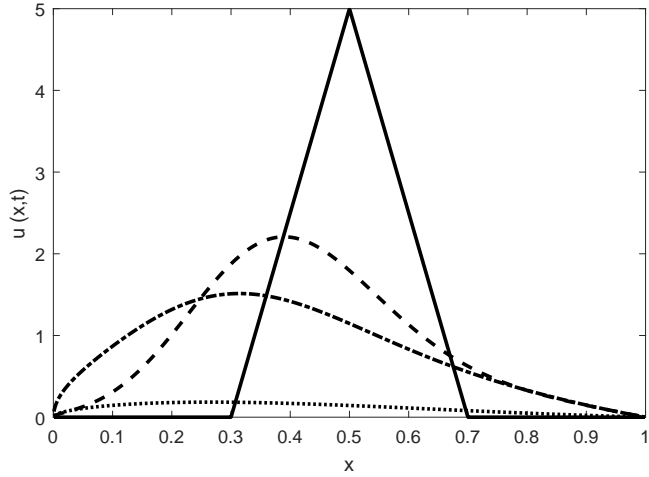}
\caption{Numerical solution of the fractional diffusion equation \eqref{DcDDeq} with $\alpha =1.5$ and $C=1$ on $0\leq x\leq 1$ with zero boundary conditions at time $t=0$ (solid line), $t=0.05$ (dashed), $t=0.1$ (dash dot), $t=0.5$ (dotted).}
\label{DcDDfig}
\end{center}
\end{figure}

Figure \ref{DcDDfig} shows a numerical solution of the fractional diffusion equation \eqref{DcDDeq} on $0\leq x\leq 1$ with zero boundary conditions \eqref{DcDDbc}.  The solution was plotted using the MATLAB routine {\tt ode15s} for stiff systems of ordinary differential equations, viewing the explicit Euler scheme \eqref{Euler0e} as the temporal discretization of a linear system of ordinary differential equations  (method of lines), with time step $\Delta t = 0.01$ and spatial grid size $h=0.001$.  The tent function initial condition
\begin{equation}
u_0(x) = \begin{cases}
25x-7.5 & \text{for $0.3 < x \leq 0.5$,}\\
-25x+17.5, & \text{for $0.5 < x < 0.7$,}\\
0 & \text{otherwise}
\end{cases}
\label{IC}
\end{equation}
satisfies the zero boundary conditions, and integrates to total mass $M=1$.  Because of the absorbing boundary conditions, solutions tend to zero as $t$ increases.  Refining the temporal or spatial discretization resulted in no visible change in the plots.  Because the fractional derivative \eqref{RLdef0} is one sided, solution curves are skewed for all $t>0$, even though the initial mass distribution is symmetric.  This can also be seen from \eqref{Euler0}, since the mass $\beta g^\alpha_0 u^k_i$ moved from state $x_i$ to state $x_{i-1}$ exceeds the total amount of mass moved to the right (which is less than $\beta (\alpha-1) u^k_i$), at any node inside the domain.

\begin{rem}
A few crucial differences from the traditional diffusion setup should be noted.  First of all, one can also characterize the physical problem as absorbing on the exterior of the open domain $0<x<1$, not just at the boundary.  Physically, mass can be displaced a long distance from the domain, and then absorbed.  Second, the form of the fractional derivative \eqref{RLdef0} also incorporates absorbing outside the domain.  The fractional diffusion equation \eqref{RL-CauchyEq} on the real line with the exterior condition $u(x,t)=0$ for $x\leq 0$ or $x\geq 1$ is equivalent to the fractional diffusion equation \eqref{DcDDeq} on the bounded domain $0\leq x\leq 1$ with zero boundary conditions \eqref{DcDDbc}.  The fractional derivative itself codes the zero exterior condition on $x<0$.  For more details, and an interesting connection to stochastic processes, see \cite{SpaceTimeFrac}. Third, since the positive Riemann-Liouville fractional derivative \eqref{RLdef} is one-sided, depending only on values of the function to the left, the zero exterior condition on $x\geq 1$ is automatically enforced.  Another way to see this is that, in the Euler scheme, mass can be transported to location $x_j$ from any location to the left, but not from the right.
\end{rem}

\section{Reflecting boundary conditions}\label{secRR}
The proper formulation of physically meaningful reflecting boundary conditions for the fractional diffusion equation \eqref{DcDDeq} on a bounded domain requires careful consideration of the nonlocal operator \eqref{RLdef0}.  Suppose that our goal is for mass leaving the domain to instead come to rest at the boundary.  Unlike the traditional diffusion setup, this mass can come from far inside the domain, not just an adjacent grid point.  Now the mass that was removed from the system in the Dirichlet model of Section \ref{SecDD} will instead be preserved, and moved to the boundary.

Let us consider the right boundary $x_n=1$, since long movements are always to the right in our setup.  For each $i=1,2,\ldots,n-1$, at each time step, mass $\beta\alpha u_i^k$ is moved out of location $x_i$, and redistributed.  Of this total, a fraction $\beta g^\alpha_{j-i+1} u_i^k$ is moved to location $x_j$ when $j=i-1$ or $j>i$.  Hence the mass landing at, or exiting the domain through, the right boundary $x_n=1$ from location $x_i$ for $i=1,2,\ldots,n-1$ is
\[\sum_{j=n}^\infty\beta g^\alpha_{j-i+1} u_i^k .\]
Then using \eqref{sum0}, along with the identity \cite[Eq.\ (20.4)]{Samko}
\begin{equation}
\label{gsum}
\sum_{j=0}^{n} g^\alpha_{j}=g^{\alpha-1}_{n},
\end{equation}
we set $b_{in}=-g^{\alpha-1}_{n-i}$ in \eqref{Euler00}, for each $i=1,2,\ldots,n-1$.
In the scheme \eqref{Euler} on the real line, the mass $\beta g^\alpha_{0} u_n^k$ moves from location $x_n$ to location $x_{n-1}$, and the remainder of the mass $\beta \alpha u_n^k$ leaving location $x_n$ moves to the right, outside the domain $0\leq x\leq 1$.  In the reflecting scheme, we retain this mass at location $x_n$ by setting $b_{nn}=-1=-g^\alpha_{0}$.

The only way that mass can move to the left boundary $x_0=0$ in this scheme is from the adjacent node $x_1=h$, hence we leave $b_{10}=g^\alpha_0=1$.  In the scheme \eqref{Euler} on the real line, mass $\beta g^\alpha_{0} u_0^k$ moves from location $x_0$ to location $x_{-1}<0$.  To prevent this, and thus to keep the scheme mass-preserving, recall that $g^\alpha_0=1$ and $g^\alpha_1=-\alpha$, and set $b_{00}=1-\alpha$.  To prevent mass leaving state $x_0$ from jumping through the right boundary, we also set
\[b_{0n}=\sum_{j=n}^\infty g^\alpha_{j+1}=-g^{\alpha-1}_{n}>0 ,\]
 and hence the explicit Euler scheme for the case of reflecting boundary conditions is written in the form \eqref{Euler00} with
 \begin{equation}
\label{EulerRRweights}
b_{ij}=\begin{cases}
g^\alpha_{j-i+1} &\text{if $0< j<n$ and $i\leq j+1$,}\\
1&\text{if $i=1$ and $j=0$,}\\
1-\alpha &\text{if $i=j=0$,}\\
-g^{\alpha-1}_{n-i}&\text{if $j=n$ and $i\leq n$,}\\
0&\text{otherwise.}
\end{cases}
\end{equation}

Next we will argue that the reflecting boundary conditions for the fractional diffusion equation \eqref{DcDDeq} on $0\leq x\leq 1$ can be written in the form
\begin{equation}
\label{nofluxBC}
{\mathbb D}^{\alpha-1}_{[0,x]}  u(0,t)={\mathbb D}^{\alpha-1}_{[0,x]}  u(1,t)=0\quad\text{for all $t\geq 0,$}
\end{equation}
using the Riemann-Liouville fractional derivative \eqref{RLdef0} of order $\alpha-1$.
When $\alpha=2$, this reduces to the classical reflecting condition $\frac{\partial}{\partial x}u(x,t)=0$ at the boundary.  First consider the right boundary $x_n=1$, and write out the iteration equation for this node: From \eqref{Euler00} and \eqref{EulerRRweights} with $\beta=C h^{-\alpha}\Delta t$ we have
$u_n^{k+1}=u_n^k-\beta g^{\alpha-1}_{n} u_0^k-\cdots-\beta g^{\alpha-1}_{1} u_{n-1}^k-\beta g^{\alpha-1}_{0} u_n^k$ which is algebraically equivalent to
\[h\frac{u_n^{k+1}-u_n^k}{\Delta t}=-Ch^{1-\alpha}\sum_{i=0}^{n} g^{\alpha-1}_{n-i} u_i^k =-Ch^{1-\alpha}\sum_{i=0}^{n} g^{\alpha-1}_{n-i} u(x_n-(n-i)h,t_k) .\]
Letting $\Delta t\to 0$ and $h\to 0$, and using the Gr\"unwald approximation \eqref{Grunwald0}, we arrive at the reflecting boundary condition \eqref{nofluxBC} at the right boundary $x=1$.

The iteration equation at the left boundary is $u_0^{k+1}=u_0^k+\beta (1-\alpha) u_0^k+\beta u_{1}^k$. Recalling that $g^{\alpha-1}_{0}=1$ and $g^{\alpha-1}_{1}=1-\alpha$, this reduces to
\[h\frac{u_0^{k+1}-u_0^k}{\Delta t}=Ch^{1-\alpha}\sum_{i=0}^{1} g^{\alpha-1}_{1-i} u_i^k ,\]
which is consistent with the reflecting boundary condition \eqref{nofluxBC} at the left boundary $x=0$.  To rigorously prove that the left boundary condition in \eqref{nofluxBC} holds, \cite{Sankaranarayanan2014} extends the matrix $h^{-\alpha}B$ by interpolation to an operator on $L^1[0,1]$, and proves convergence to the generator \eqref{RLdef0} with boundary conditions \eqref{nofluxBC}.

\begin{rem}\label{FluxRemark}
The reflecting boundary conditions \eqref{nofluxBC} can be seen as zero flux conditions at the boundary:  Note that the fractional diffusion equation \eqref{DcDDeq} can be derived from the traditional conservation of mass equation
\begin{equation}
\label{COM}
\frac{\partial }{\partial t} u(x,t)=-\frac{\partial }{\partial x} q(x,t)
\end{equation}
together with the flux equation (or fractional Fick's Law, see \cite{Schumer2001})
\begin{equation}
\label{fracFickRL}
q(x,t)=-C{\mathbb D}^{\alpha-1}_{[0,x]}  u(x,t)=-C\frac{\partial}{\partial x} \frac{1}{\Gamma(2-\alpha)} \int_{0}^x u(y,t)(x-y)^{1-\alpha} dy .
\end{equation}
Hence \eqref{nofluxBC} simply sets the flux to zero at the boundary.  When $\alpha=2$, the fractional Fick's Law reduces to the traditional Fick's Law $q(x,t)=-C\frac{\partial}{\partial x}u(x,t)$.
\end{rem}

\begin{figure}
\begin{center}
\hskip-0.5in
\includegraphics[width=4in]{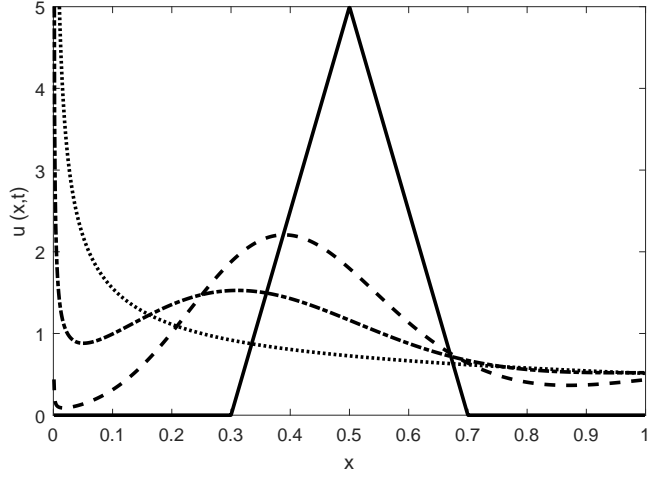}
\caption{Numerical solution of the fractional diffusion equation \eqref{DcDDeq} with $\alpha =1.5$ and $C=1$ on $0\leq x\leq 1$ with reflecting boundary conditions \eqref{nofluxBC} at time $t=0$ (solid line), $t=0.05$ (dashed), $t=0.1$ (dash dot), $t=0.5$ (dotted). }
\label{DNNfig}
\end{center}
\end{figure}

Figure \ref{DNNfig} shows a numerical solution of the fractional diffusion equation \eqref{DcDDeq} on $0\leq x\leq 1$ with reflecting boundary conditions, using the same numerical method and initial function as in Figure \ref{DcDDfig}.  As in Figure \ref{DcDDfig}, and for the same reason, solution curves are skewed for all $t>0$, even though the initial mass distribution is symmetric.  However, there is a profound difference in the solutions.  Here the total mass (area under the curve) remains equal to the initial mass $M=1$ for all $t>0$, because of the reflecting boundary conditions.  As $t$ increases, the solutions approach the steady state solution $u(x)=(\alpha-1)x^{\alpha-2}$ on $0<x<1$.

\begin{rem}\label{theoryNNrem}
In \cite{Sankaranarayanan2014} it is shown that the Cauchy problem \eqref{DcDDeq} with reflecting boundary conditions \eqref{nofluxBC} is well-posed on the Banach space $L^1[0,1]$, and the exact domain of the generator is computed.
\end{rem}

\begin{rem}\label{SteadyRL}
The general steady state solution to the fractional diffusion equation \eqref{DcDDeq} is $u(x)=c_1 x^{\alpha-1}+ c_2x^{\alpha-2}$ where $c_1,c_2$ are arbitrary real numbers.  To see this, note that the Riemann-Liouville fractional derivative ${\mathbb D}^\alpha_{[0,x]}  u(x)=\frac{d^2}{dx^2}{\mathbb J}^{2-\alpha}_{[0,x]}  u(x)$ where the Riemann-Liouville fractional integral
\begin{equation}
\label{RLintDef0}
{\mathbb J}^{\gamma}_{[0,x]}   u(x) = \frac{1}{\Gamma(\gamma)} \int_{0}^x u(y,t)(x-y)^{\gamma-1} dy
\end{equation}
for any $\gamma>0$.  Using the general formula (e.g., see \cite[Example 2.7]{FCbook})
\begin{equation}
\label{RLintPower}
{\mathbb J}^{\gamma}_{[0,x]}  [x^p]=  \frac{\Gamma(p+1)}{\Gamma(p+\gamma+1)}x^{p+\gamma}
\end{equation}
we see that
\[{\mathbb J}^{2-\alpha}_{[0,x]}  u(x)= c_1\Gamma(\alpha)x+c_2\Gamma(\alpha-1) .\]
Then ${\mathbb D}^\alpha_{[0,x]}  u(x)=0$ for all $0<x<1$. The only steady state solution with total mass 1 that satisfies the reflecting boundary conditions \eqref{nofluxBC} has $c_1=0$ and $c_2=\alpha-1$.   The only steady state solution that satisfies the absorbing boundary conditions \eqref{nofluxBC} has $c_1=0$ and $c_2=0$.
\end{rem}

\section{Absorbing on one side, reflecting on the other}
 Next we consider the fractional diffusion equation \eqref{DcDDeq} on $0\leq x\leq 1$ with a reflecting boundary condition on the left, and an absorbing boundary condition on the right:
 \begin{equation}
\label{caseRA}
{\mathbb D}^{\alpha-1}_{[0,x]}  u(0,t)=0\quad\text{and}\quad u(1,t)=0\quad\text{for all $t\geq 0\,$}
\end{equation}
The explicit Euler scheme for this problem is \eqref{Euler00} with
\begin{equation}
\label{EulerRAweights}
b_{ij}=\begin{cases}
g^\alpha_{j-i+1} &\text{if $0< j<n$ and $i\leq j+1$,}\\
1&\text{if $i=1$ and $j=0$,}\\
1-\alpha &\text{if $i=j=0$,}\\
0&\text{otherwise.}
\end{cases}
\end{equation}
This combines the reflecting boundary condition at $x_0=0$ from \eqref{EulerRRweights} and the absorbing boundary condition at $x_0=1$ from \eqref{EulerDDweights}.

Figure \ref{figND} shows the resulting numerical solution of the fractional diffusion equation \eqref{DcDDeq} on $0\leq x\leq 1$ with boundary conditions \eqref{caseRA}, using the same numerical method and initial function as in Figure \ref{DcDDfig}.  The solutions are skewed to the right, and approach the steady state solution $u=0$ as $t$ increases.  In this model, mass accumulates at the reflecting boundary $x=0$, but then will eventually be absorbed at the right boundary $x=1$.

\begin{figure}
\begin{center}
\hskip-0.5in
\includegraphics[width=4in]{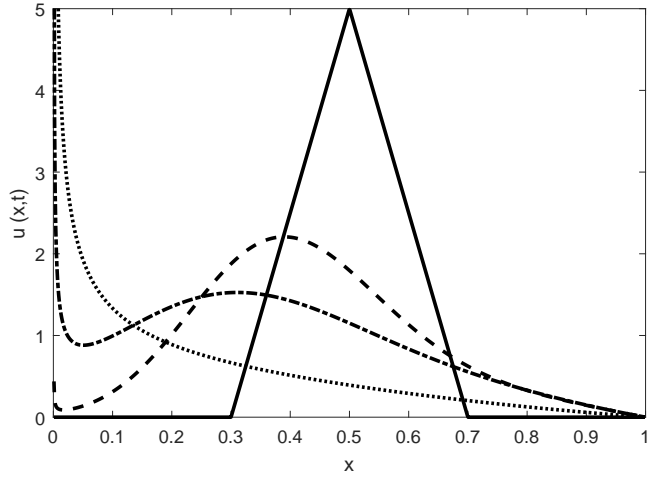}
\caption{Numerical solution of the fractional diffusion equation \eqref{DcDDeq} with $\alpha =1.5$ and $C=1$ on $0\leq x\leq 1$ with boundary conditions \eqref{caseRA}: Reflecting on the left, and absorbing on the right at time $t=0$ (solid line), $t=0.05$ (dashed), $t=0.1$ (dash dot), $t=0.5$ (dotted). }
\label{figND}
\end{center}
\end{figure}

Next we consider the opposite case, the fractional diffusion equation \eqref{DcDDeq} on $0\leq x\leq 1$ with an absorbing boundary condition on the left, and a reflecting boundary condition on the right:
\begin{equation}
\label{caseAR}
u(0,t)=0\quad\text{and}\quad{\mathbb D}^{\alpha-1}_{[0,x]}  u(1,t)=0\quad\text{for all $t\geq 0.$}
\end{equation}
The explicit Euler scheme for this problem is \eqref{Euler00} with
 \begin{equation}
\label{EulerARweights}
b_{ij}=\begin{cases}
g^\alpha_{j-i+1} &\text{if $0< j<n$ and $i\leq j+1$,}\\
-g^{\alpha-1}_{n-i}&\text{if $j=n$ and $i\leq n$,}\\
0&\text{otherwise.}
\end{cases}
\end{equation}
This combines the absorbing boundary condition at $x_0=0$ from \eqref{EulerDDweights} and the reflecting boundary condition at $x_0=1$ from \eqref{EulerRRweights}.

Figure \ref{figDN} shows the resulting numerical solution of the fractional diffusion equation \eqref{DcDDeq} on $0\leq x\leq 1$ with boundary conditions \eqref{caseAR}, using the same numerical method and initial function as in Figure \ref{DcDDfig}.  The solutions are skewed to the right, and approach the steady state solution $u=0$ as $t$ increases.  In this model, mass is reflected at the right boundary, and then eventually absorbed at the left boundary.

\begin{figure}
\begin{center}
\hskip-0.5in
\includegraphics[width=4in]{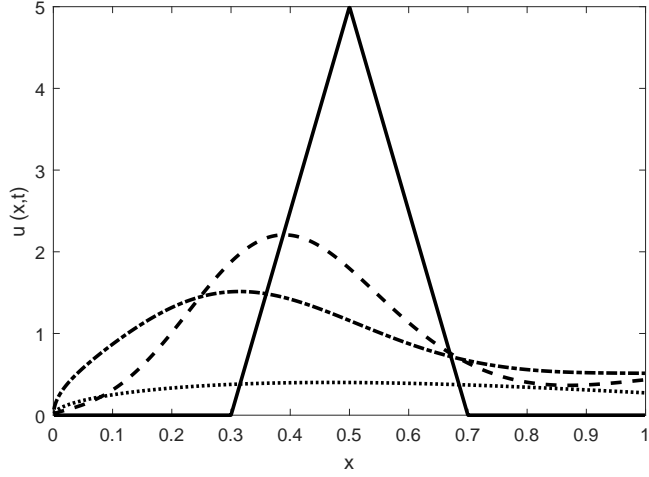}
\caption{Numerical solution of the fractional diffusion equation \eqref{DcDDeq} with $\alpha =1.5$ and $C=1$ on $0\leq x\leq 1$ with boundary conditions \eqref{caseAR}: Reflecting on the right, and absorbing on the left at time $t=0$ (solid line), $t=0.05$ (dashed), $t=0.1$ (dash dot), $t=0.5$ (dotted). }
\label{figDN}
\end{center}
\end{figure}

\begin{rem}\label{theoryNDDNrem}
In \cite{Sankaranarayanan2014} it is shown that the Cauchy problem \eqref{DcDDeq} on $0\leq x\leq 1$ with boundary conditions \eqref{caseRA} or \eqref{caseAR} is well-posed on the Banach space $L^1[0,1]$, and the exact domain of the generator is computed.  Then it follows as in Remark \ref{theoryDDrem} that any stable and consistent numerical method converges to the unique solution.
\end{rem}

\section{Caputo fractional flux}
An alternative to the fractional diffusion equation \eqref{DcDDeq} is the Caputo fractional flux model.  The Caputo fractional derivative is defined by
\begin{equation}
\label{Cdef0}
\partial^\gamma_{[0,x]}  u(x) = \frac{1}{\Gamma(n-\gamma)} \int_0^x u^{(n)}(y)(x-y)^{n-\gamma-1} dy
\end{equation}
for $\gamma>0$ and $n-1<\gamma\leq n$, where $u^{(n)}(x)$ is the $n$th derivative.  It differs from \eqref{RLdef0} in that the derivative is moved inside the integral.  These two fractional derivatives are not equivalent.  For example,
\begin{equation}
\label{RLtoCaputo}
\partial^\gamma_{[0,x]}  u(x) ={\mathbb D}^\gamma_{[0,x]}  u(x) -u(0)\frac{x^{-\gamma}}{\Gamma(1-\gamma)}
\end{equation}
when $0<\gamma<1$ \cite[Eq.\ (2.33)]{FCbook}.  Recall from Remark \ref{FluxRemark} that the fractional diffusion equation \eqref{DcDDeq} can be derived from the conservation of mass equation \eqref{COM} and the Riemann-Liouville fractional Fick's Law \eqref{fracFickRL}.  Noting that Fick's Law is purely empirical, we can instead consider the Caputo fractional flux
\begin{equation}
\label{fracFickC}
q(x,t)=-C\partial^{\alpha-1}_{[0,x]}  u(x,t)=-\frac{C}{\Gamma(2-\alpha)} \int_{0}^x u'(y,t)(x-y)^{1-\alpha} dy ,
\end{equation}
where $u'(x,t)$ denotes the $x$ derivative.  This leads to the fractional diffusion equation with Caputo flux:
\begin{equation}
\label{C-CauchyEq0}
\frac{\partial}{\partial t}u(x,t)= C\, {\mathbf D}^\alpha_{[0,x]}  u(x) ,
\end{equation}
where the {\it Patie-Simon fractional derivative} is defined by
\begin{equation}
\label{PSdvt}
{\mathbf D}^\alpha_{[0,x]}  u(x) =\frac{1}{\Gamma(2-\alpha)} \frac{\partial}{\partial x}\int_0^x \frac{\partial}{\partial x} u(x-y)y^{-\alpha} dy
\end{equation}
for $1<\alpha<2$.
Patie and Simon \cite[p.\ 570]{PatieSimon} showed that this operator is the (backward) generator of a standard spectrally negative $\alpha$-stable process reflected to stay positive \cite[p.\ 573]{PatieSimon}.
Use the relation \eqref{RLtoCaputo} and the definition \eqref{RLdef0} to see that
\begin{equation}\begin{split}\label{PStoRL}
{\mathbf D}^\alpha_{[0,x]}  f(x) &=\frac{d}{dx}\left[\partial^{\alpha-1}_{[0,x]}  f(x)\right]\\
&=\frac{d}{dx} \left[{\mathbb D}^{\alpha-1}_{[0,x]}  f(x) -f(0)\frac{x^{1-\alpha}}{\Gamma(2-\alpha)}\right]\\
&={\mathbb D}^{\alpha}_{[0,x]}  f(x)-f(0)\frac{x^{-\alpha}}{\Gamma(1-\alpha)}
\end{split}\end{equation}
which relates the two derivatives when $1<\alpha<2$.

A Gr\"unwald finite difference scheme for the fractional derivative \eqref{PSdvt} can be written as
\begin{equation}
\label{GrunwaldPS}
{\mathbf D}^\alpha_{[0,x]}  f(x) =\lim_{h\to 0} h^{-\alpha}\left[\sum_{i=0}^{j+1} g^\alpha_i f(x-(i-1)h)- g^{\alpha-1}_{j+1} f(x-jh)\right]
\end{equation}
where $j=j(h)=[x/h]+1$.  To see this, apply \cite[Proposition 2.1]{FCbook} to see that the first term in \eqref{GrunwaldPS} converges to ${\mathbb D}^{\alpha}_{[0,x]}  f(x)$.  Then note that $x-jh\to 0$ as $h\to 0$, and that \cite[Eq.\ (2.5)]{FCbook}
\begin{equation}\label{GrunwaldWeightsAsy}
g^{\alpha-1}_j\sim \frac{1-\alpha}{\Gamma(2-\alpha)} j^{-\alpha}\quad\text{as $j\to\infty$,}
\end{equation}
meaning that the ratio between the left and right terms tends to 1 as $j\to\infty$.  Then
\begin{equation*}\begin{split}
h^{-\alpha}g^{\alpha-1}_{j+1}&\sim h^{-\alpha} \frac{1-\alpha}{\Gamma(2-\alpha)} ([x/h]+1)^{-\alpha}\\
&\sim \frac{1-\alpha}{\Gamma(2-\alpha)} x^{-\alpha} = \frac{x^{-\alpha}}{\Gamma(1-\alpha)}
\end{split}\end{equation*}
using $\Gamma(z+1)=z\Gamma(z)$.  Then \eqref{GrunwaldPS} follows using \eqref{PStoRL}.

Next we construct an explicit Euler scheme \eqref{Euler00} for the fractional diffusion equation with Caputo flux \eqref{C-CauchyEq0}.  For $x=x_j=x-jh$ and $t_k=k\Delta t$, note that the Gr\"unwald approximation of the Patie-Simon fractional derivative is
\[{\mathbf D}^\alpha_{[0,x]}  u^k_j\approx h^{-\alpha}\left[\sum_{i=0}^{j+1} g^\alpha_{j-i+1} u^k_i- g^{\alpha-1}_{j+1} u_0^k\right]=h^{-\alpha}\sum_{i=0}^{n} b_{ij} u^k_i\]
where $b_{0j}=g^\alpha_{j+1}-g^{\alpha-1}_{j+1}$, $b_{ij}=g^\alpha_{j-i+1}$ for $0<i\leq j+1$, and $b_{ij}=0$ for $i>j+1$.  Hence the only change in the iteration matrix $B=[b_{ij}]$ is in the top row.
From \eqref{gsum} it follows easily that
\begin{equation}
\label{gdiff}
g^\alpha_{n}-g^{\alpha-1}_{n}=-g^{\alpha-1}_{n-1} ,
\end{equation}
and hence we can write $b_{0j}=-g^{\alpha-1}_{j}$.
Now in order to solve the fractional diffusion equation with Caputo flux \eqref{C-CauchyEq0}, we need only to enforce appropriate boundary conditions.

First assume zero boundary conditions.  As in \eqref{EulerDDweights} it is sufficient to set $b_{ij}=0$ for $j=0$ or $j=n$, since a mass proportional to $b_{ij}$ is transported from location $x_i$ to location $x_j$, and we want this mass to vanish.  Then, we obtain the explicit Euler scheme \eqref{Euler00} with weights
\begin{equation}
\label{EulerDDweightsPS}
b_{ij}=\begin{cases}
g^\alpha_{j-i+1} &\text{if $0<j<n$ and $0<i\leq j+1$,}\\
-g^{\alpha-1}_{j}&\text{if $i=0$ and $0< j< n$,}\\
0&\text{otherwise.}
\end{cases}
\end{equation}
The iteration matrix $B=[b_{ij}]$ differs from \eqref{EulerDDweights} only in the first row $i=0$.  Since the mass at the left endpoint $x_0=0$ is always zero in this case, there is no difference in the solutions, and hence Figure \ref{DcDDfig} is also the solution to the fractional diffusion equation with Caputo flux \eqref{C-CauchyEq0} and absorbing boundary conditions \eqref{DcDDbc}.  In fact, since we assume a zero boundary condition on the left, $u(0,t)=0$ for all $t>0$, the fractional diffusion equation with Caputo flux \eqref{C-CauchyEq0} and the Riemann-Liouville fractional diffusion equation \eqref{DcDDeq} on $0\leq x\leq 1$ are equivalent, due to the relation \eqref{PStoRL}.

Next consider a reflecting boundary condition on both sides. Since the iteration matrix $B=[b_{ij}]$ has not changed except in the first row, the argument in Section \ref{secRR} applies for every state $x_j$ with $j>0$, i.e., we set $b_{jn}=-g^{\alpha-1}_{n-i}$ for all $i=1,2,\ldots,n$ as in \eqref{EulerRRweights}.  As for the first row, the only way that mass can move to the left boundary $x_0=0$ in this scheme is from the adjacent node $x_1=h$, hence we leave $b_{10}=g^\alpha_0=1$.  In the scheme \eqref{Euler00} for the Patie-Simon fractional derivative, we have $b_{00}=-g^{\alpha-1}_{0}=-1$.  To prevent mass from state $x_0=0$ jumping through the right boundary $x_n=1$, since $b_{0j}=-g^{\alpha-1}_{j}$ for $j=1,2,\ldots,n-1$, and since we require $\sum_{j}b_{ij}=0$ for a mass-preserving scheme, we must set
\[b_{0n}=-\sum_{j=0}^{n-1} b_{0j}=1+\sum_{j=1}^{n-1} g^{\alpha-1}_{j}=\sum_{j=0}^{n-1} g^{\alpha-1}_{j}=g^{\alpha-2}_{n-1}\]
using \eqref{gsum}.
Hence the explicit Euler scheme for this problem is \eqref{Euler00} with
 \begin{equation}
\label{EulerRRweightsPS}
b_{ij}=\begin{cases}
g^\alpha_{j-i+1} &\text{if $0< j<n$ and $0<i\leq j+1$,}\\
1&\text{if $i=1$ and $j=0$,}\\
-1 &\text{if $i=j=0$,}\\
-g^{\alpha-1}_{j}&\text{if $i=0$ and $0< j< n$,}\\
g^{\alpha-2}_{n-1}&\text{if $j=n$ and $i=0$,}\\
-g^{\alpha-1}_{n-i}&\text{if $j=n$ and $0<i\leq n$,}\\
0&\text{otherwise.}
\end{cases}
\end{equation}

Next we will argue that the reflecting boundary conditions for the fractional diffusion equation with Caputo flux \eqref{C-CauchyEq0} on $0\leq x\leq 1$ can be written in the form
\begin{equation}
\label{PSnofluxBC}
\partial^{\alpha-1}_{[0,x]}  u(0,t)=\partial^{\alpha-1}_{[0,x]}  u(1,t)=0\quad\text{for all $t\geq 0$,}
\end{equation}
using the Caputo derivative \eqref{Cdef0}.  That is, the reflecting boundary conditions zero out the Caputo flux at the boundary.  First consider the right boundary $x_n=1$, and write out the iteration equation for this node: From \eqref{Euler00} and \eqref{EulerRRweightsPS} with $\beta=C h^{-\alpha}\Delta t$ we have \[u_n^{k+1}=u_n^k-\beta g^{\alpha-2}_{n-1} u_0^k-\beta g^{\alpha-1}_{n-1} u_1^k-\cdots-\beta g^{\alpha-1}_{1} u_{n-1}^k-\beta g^{\alpha-1}_{0}u_n^k .\]
Using \eqref{gdiff} this is equivalent to
\[h\frac{u_n^{k+1}-u_n^k}{\Delta t}=-Ch^{1-\alpha}\left[\sum_{i=0}^{n} g^{\alpha-1}_{n-i} u_i^k-g^{\alpha-2}_{n} u_0^k\right] .\]
Letting $\Delta t\to 0$ and $h\to 0$, the left-hand side converges to zero, the first term on the right converges to ${\mathbb D}^{\alpha-1}_{[0,x]}  u(1,t)$ using the Gr\"unwald approximation \eqref{Grunwald0}, and recalling that $hn=1$, the second term
\[Ch^{1-\alpha}g^{\alpha-2}_{n} u_0^k\sim Cn^{\alpha-1}\frac{2-\alpha}{\Gamma(3-\alpha)} n^{1-\alpha}u_0^k\to \frac{C}{\Gamma(2-\alpha)} u(0,t)\]
as $h\to 0$ using \eqref{GrunwaldWeightsAsy}.  Using \eqref{RLtoCaputo} with $\gamma=\alpha-1$ and $x=1$, it follows that the entire right-hand side converges to the Caputo derivative of order $\alpha-1$, and hence the reflecting boundary condition \eqref{PSnofluxBC} holds at the right boundary $x=1$.

Using $b_{0j}=g^\alpha_{j+1}-g^{\alpha-1}_{j+1}$, the iteration equation at the left boundary is $u_0^{k+1}=u_0^k+\beta g^\alpha_{1} u_0^k-\beta g^{\alpha-1}_{1}u_0^k+\beta g^{\alpha-1}_{0} u_{1}^k$. This reduces to
\[h\frac{u_n^{k+1}-u_n^k}{\Delta t}=-Ch^{1-\alpha}\left[\sum_{i=0}^{1} g^{\alpha-1}_{1-i} u_i^k-g^{\alpha-2}_{1} u_0^k\right]\]
which is consistent with the reflecting boundary condition \eqref{PSnofluxBC} at the left boundary $x=0$.  A rigorous proof that the left boundary condition in \eqref{PSnofluxBC} holds is similar to the case of the Riemann-Liouville generator, see \cite{Sankaranarayanan2014}.

\begin{rem}
Comparing \eqref{nofluxBC} and \eqref{PSnofluxBC} shows that the form of the reflecting boundary condition {\it also changes} when we change the type of fractional derivative in the fractional diffusion equation.  When $\alpha=2$, both forms reduce to the classical reflecting boundary condition $\frac{\partial}{\partial x}u(0,t)=\frac{\partial}{\partial x}u(1,t)=0$ .
\end{rem}

\begin{figure}
\begin{center}
\hskip-0.5in
\includegraphics[width=4in]{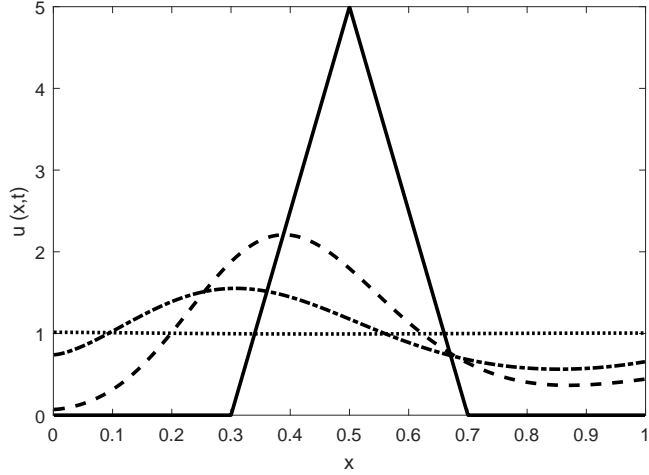}
\caption{Numerical solution of the fractional diffusion equation with Caputo flux \eqref{C-CauchyEq0} with $\alpha =1.5$ and $C=1$ on $0\leq x\leq 1$ with reflecting boundary conditions \eqref{PSnofluxBC} at time $t=0$ (solid line), $t=0.05$ (dashed), $t=0.1$ (dash dot), $t=0.5$ (dotted). }
\label{DcNNfig}
\end{center}
\end{figure}

Figure \ref{DcNNfig} shows a numerical solution of the fractional diffusion equation with Caputo flux \eqref{C-CauchyEq0} on $0\leq x\leq 1$ with reflecting boundary conditions, using the same numerical method and initial function as in Figure \ref{DcDDfig}.  Solution curves are skewed for $0<t<\infty$, and the total mass remains equal to the initial mass $M=1$ for all $t>0$, since the scheme is mass-preserving.  As $t$ increases, the solutions approach the unique steady state solution $u(x)=1$ on $0\leq x\leq 1$ with unit mass, which is very different than the unit mass steady state solution $u(x)=(\alpha-1)x^{\alpha-2}$ to the fractional diffusion equation \eqref{DcDDeq} on $0\leq x\leq 1$ with reflecting boundary conditions.  In the present case, the steady state solution is easy to verify, by simply plugging in to \eqref{C-CauchyEq0}.  In \cite{Sankaranarayanan2014} it is shown that the Cauchy problem \eqref{C-CauchyEq0} with reflecting boundary conditions \eqref{nofluxBC} is well-posed on the Banach space $L^1[0,1]$, and the exact domain of the generator is computed.


Next we consider the fractional diffusion equation with Caputo flux \eqref{C-CauchyEq0} on $0\leq x\leq 1$ with boundary conditions
  \begin{equation}
\label{PScaseRA}
\partial^{\alpha-1}_{[0,x]}  u(0,t)=0\quad\text{and}\quad u(1,t)=0\quad\text{for all $t\geq 0,$}
\end{equation}
reflecting at the left boundary $x=0$ and absorbing at the right boundary $x=1$.  Here we simply zero out the coefficients $b_{ij}$ from \eqref{EulerRRweightsPS} governing mass transport from state $i$ to state $j=n$.  This yields the explicit Euler scheme \eqref{Euler00} with
 \begin{equation}
\label{EulerRAweightsPS}
b_{ij}=\begin{cases}
g^\alpha_{j-i+1} &\text{if $0< j<n$ and $0<i\leq j+1$,}\\
1&\text{if $i=1$ and $j=0$,}\\
-1 &\text{if $i=j=0$,}\\
-g^{\alpha-1}_{j}&\text{if $i=0$ and $0< j< n$,}\\
0&\text{otherwise.}
\end{cases}
\end{equation}

Figure \ref{DcNDfig} shows a numerical solution of the fractional diffusion equation with Caputo flux \eqref{C-CauchyEq0} on $0\leq x\leq 1$ with boundary conditions \eqref{PScaseRA}, using the same numerical method and initial function as in Figure \ref{DcDDfig}.  Solution curves are skewed for all $t>0$, and approach the unique steady state solution $u(x)=0$ on $0\leq x\leq 1$ as $t$ increases.  In \cite{Sankaranarayanan2014} it is shown that the Cauchy problem \eqref{C-CauchyEq0} with these boundary conditions \eqref{PScaseRA} is well-posed on the Banach space $L^1[0,1]$, and the exact domain of the generator is computed.

\begin{figure}
\begin{center}
\hskip-0.5in
\includegraphics[width=4in]{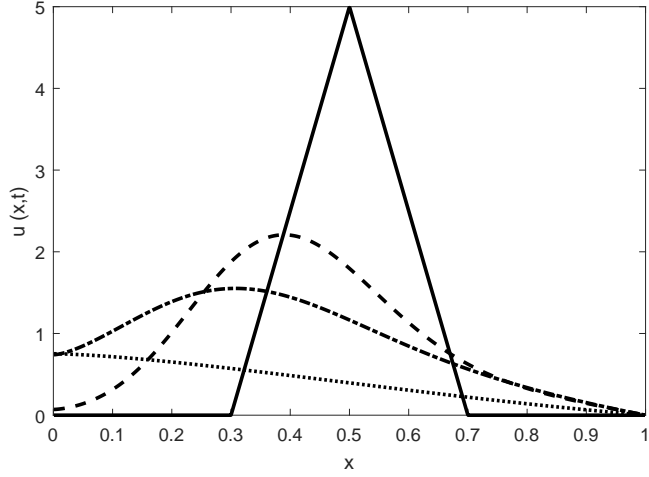}
\caption{Numerical solution of the fractional diffusion equation with Caputo flux \eqref{C-CauchyEq0} with $\alpha =1.5$ and $C=1$ on $0\leq x\leq 1$ with boundary conditions \eqref{PScaseRA}: reflecting on the left, and absorbing on the right at time $t=0$ (solid line), $t=0.05$ (dashed), $t=0.1$ (dash dot), $t=0.5$ (dotted). }
\label{DcNDfig}
\end{center}
\end{figure}

\begin{rem}
The {\it left} reflecting boundary condition for the fractional diffusion equation with Caputo flux \eqref{C-CauchyEq0} can also be written in the traditional form
\begin{equation}
\label{remBCeq}
\frac{\partial}{\partial x}u(0,t)=0\quad\text{for all $t\geq 0$.}
\end{equation}
To see this, rewrite the iteration equation $u_0^{k+1}=u_0^k-\beta u_0^k+\beta u_{1}^k$ in the equivalent form
\[h^{\alpha-1}\frac{u_n^{k+1}-u_n^k}{\Delta t}=C\left[\frac{u_1^k-u_0^k}{h}\right]\]
and let $h\to 0$ and $\Delta t\to 0$.  In view of \eqref{Cdef0}, the condition \eqref{remBCeq} also implies that the Caputo fractional derivative $\partial^{\alpha-1}_{[0,x]}  u(0,t)=0$.  The same is {\it not} true of the Riemann-Liouville derivative, and indeed, even the steady state solution $u(x)=(\alpha-1)x^{\alpha-2}$ of the Riemann-Liouville fractional diffusion equation \eqref{DcDDeq} with reflecting boundary conditions does not satisfy the condition \eqref{remBCeq}.  The nonlocal {\it right} reflecting boundary condition for the fractional diffusion equation with Caputo flux \eqref{C-CauchyEq0} cannot be reduced to a local first derivative condition, since it depends on the values of the solution across the entire domain.  Indeed, one can see in Figure \ref{DcNNfig} that $\frac{\partial}{\partial x}u(1,t)\neq 0$ at the right boundary.
\end{rem}

Finally we consider the fractional diffusion equation with Caputo flux \eqref{C-CauchyEq0} on $0\leq x\leq 1$ with boundary conditions
  \begin{equation}
\label{PScaseAR}
 u(0,t)=0\quad\text{and}\quad \partial^{\alpha-1}_{[0,x]} u(1,t)=0\quad\text{for all $t\geq 0,$}
\end{equation}
absorbing at the left boundary $x=0$ and reflecting at the right boundary $x=1$.  Since $u(0,t)=0$ for all $t>0$, this problem is mathematically equivalent to the fractional diffusion equation \eqref{DcDDeq} on $0\leq x\leq 1$ with boundary conditions \eqref{caseAR}.  Hence Figure \ref{figDN} also represents the solution to this fractional boundary value problem.  One can also see this by setting $b_{i0}=0$ in \eqref{EulerRRweightsPS} to get
\begin{equation}
\label{EulerARweightsPS}
b_{ij}=\begin{cases}
g^\alpha_{j-i+1} &\text{if $0< j<n$ and $0<i\leq j+1$,}\\
-g^{\alpha-1}_{j}&\text{if $i=0$ and $0< j< n$,}\\
-g^{\alpha-2}_{n-1}&\text{if $j=n$ and $i=0$,}\\
-g^{\alpha-1}_{n-i}&\text{if $j=n$ and $0<i\leq n$,}\\
0&\text{otherwise.}
\end{cases}
\end{equation}
Since $u_0^k=0$ for all $k$, the first row of the iteration matrix $B$ is immaterial, and the rest of the matrix is exactly the same as for the corresponding case of the fractional diffusion equation \eqref{DcDDeq}.

\begin{rem}
The general steady state solution to the fractional diffusion equation with Caputo flux \eqref{C-CauchyEq0} is $u(x)=c_1 x^{\alpha-1}+c_2$.  This can be verified by a calculation similar to Remark \ref{SteadyRL}:  Since ${\mathbf D}^\alpha_{[0,x]}  u(x)=\frac{d}{dx}{\mathbb J}^{2-\alpha}_{[0,x]}  u'(x)$, we have
\[{\mathbf D}^\alpha_{[0,x]} u(x)=\frac{d}{dx}{\mathbb J}^{2-\alpha}_{[0,x]} c_1 (\alpha-1)x^{\alpha-2}=\frac{d}{dx}[c_1 (\alpha-1)\Gamma(\alpha-1)]= 0\]
using \eqref{RLintPower}.  Zero boundary conditions require $c_2=0$ to make $u(0)=0$, and then also $c_1=0$ to make $u(1)=0$.  Hence the unique steady state  solution satisfying these boundary conditions is $u(x)=0$.  For reflecting boundary conditions, we compute
\begin{equation*}\begin{split}
\partial^{\alpha-1}_{[0,x]} u(x)&={\mathbb J}^{2-\alpha}_{[0,x]} u'(x)=c_1\Gamma(\alpha)\\
\end{split}\end{equation*}
for $0<x<1$.  The right boundary condition $\partial^{\alpha-1}_{[0,x]} u(1)=0$ requires $c_1=0$, and then the left boundary condition is satisfied for any real number $c_2$.  Take $c_2=1$ to get the solution with total mass 1.  If just the left boundary condition is absorbing, we require $c_2=0$, and then the reflecting boundary condition on the right requires $c_1=0$ as well.  If just the right boundary condition is absorbing, then $c_1 +c_2=0$.  Then if the left boundary is reflecting, $c_1=0$, and hence $c_2=0$ as well.
\end{rem}

\begin{rem}
Cushman and Ginn \cite{CushmanGinn} use the fractional derivative \eqref{C-CauchyEq0} (on the real line, with the lower integration limit $0$ changed to $-\infty$) to model contaminant transport in groundwater.  For such problems, all three fractional derivatives are equivalent, since the boundary term at $x=-\infty$ vanishes.
\end{rem}

\begin{rem}
In \cite{ReflectedStable} we show that the backward generator of a standard spectrally negative $\alpha$-stable process reflected to stay positive is the Caputo fractional derivative \eqref{Cdef0}.  Since \cite[Eq.\ (1.2)]{PatieSimon}
\begin{equation}
\label{PStoCaputo}
\partial^\alpha_{[0,x]}  u(x)  ={\mathbf D}^\alpha_{[0,x]}  u(x) -u'(0)\frac{x^{1-\alpha}}{\Gamma(2-\alpha)}
\end{equation}
for $1<\alpha<2$, and since $u'(0)=0$ for every function in the domain of the generator \cite[Remark 2.3 (a)]{PatieSimon}, these two forms are completely equivalent.
\end{rem}

\begin{rem}
For either the Riemann-Liouville fractional diffusion equation \eqref{RL-CauchyEq} or the fractional diffusion equation with Caputo flux
\eqref{C-CauchyEq0}, the fractional derivative operator with a zero boundary condition at one or both boundaries is invertible \cite{Sankaranarayanan2014}.  This implies that, for any initial data $u_0(x)$, the solution converges to the unique steady state solution $u=0$, see Appendix for details.  In the case of reflecting boundary conditions, the numerical evidence suggests convergence, but we do not have a proof.
\end{rem}

\section{What can go wrong}
One could also consider the Caputo fractional differential equation
\begin{equation}
\label{CaputoFDE}
\frac{\partial}{\partial t}u(x,t)= C\, \partial^\alpha_{[0,x]}  u(x)
\end{equation}
with $1<\alpha<2$ on the unit interval $0\leq x\leq 1$, using the Caputo fractional derivative \eqref{Cdef0}.   However, solutions to \eqref{CaputoFDE} are not positivity preserving.  An explicit Euler scheme to solve this problem can be developed using the Caputo Gr\"unwald formula
\begin{equation}\begin{split}
\label{GrunwaldCaputo}
\partial^\alpha_{[0,x]}  f(x) =\lim_{h\to 0} h^{-\alpha}\Big[&\sum_{i=0}^{j+1} g^\alpha_i f(x-(i-1)h)- g^{\alpha-1}_{j+1} f(x-jh)\\
&- g^{\alpha-2}_{j+1} f(x-(j-1)h)+g^{\alpha-2}_{j+1} f(x-jh)\Big]
\end{split}\end{equation}
where $j=j(h)=[x/h]+1$.  The proof that \eqref{GrunwaldCaputo} holds is very similar to \eqref{GrunwaldPS}.  This leads to the explicit Euler scheme \eqref{Euler00} with
\begin{equation}
\label{EulerDDweightsCaputo1}
b_{ij}=\begin{cases}
g^\alpha_{j-i+1} &\text{if $0<j<n$ and $1<i\leq j+1$,}\\
-g^{\alpha-1}_{j}+g^{\alpha-2}_{j+1}&\text{if $i=0$ and $0<j<n$,}\\
g^\alpha_j - g^{\alpha-2}_{j+1}&\text{if $i=1$ and $0<j<n$,}\\
0&\text{otherwise.}
\end{cases}
\end{equation}
Figure \ref{Exfig1} shows a numerical solution of the fractional differential equation \eqref{CaputoFDE} on $0\leq x\leq 1$ with Dirichlet boundary conditions \eqref{DcDDbc}, and initial function
\begin{equation}
u_0(x) = \begin{cases}
\displaystyle{\frac{64 \pi^3}{\pi^2 -4} \left(x-\tfrac 14\right)^2 \sin(4 \pi x)} & \text{for $0<x<0.25$,}\\
0 & \text{otherwise,}
\end{cases}
\label{IC2}
\end{equation}
using the same numerical method as in Figure \ref{DcDDfig}.
Since the solution takes negative values with nonnegative initial data, the Caputo fractional differential equation \eqref{CaputoFDE} cannot provide a physically meaningful model for anomalous diffusion.

\begin{figure}
\begin{center}
\hskip-0.5in
\includegraphics[width=4in]{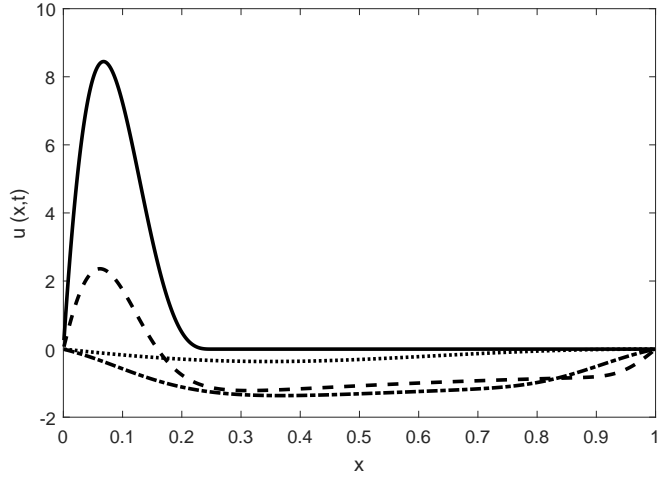}
\caption{Numerical solution of the Caputo fractional differential equation \eqref{CaputoFDE} with $\alpha =1.5$ and $C=1$ on $0\leq x\leq 1$ with zero boundary conditions \eqref{DcDDbc} at time $t=0$ (solid line), $t=0.01$ (dashed), $t=0.04$ (dash dot), $t=0.2$ (dotted).  Solutions takes negative values even though the initial function \eqref{IC2} is nonnegative.}
\label{Exfig1}
\end{center}
\end{figure}

\begin{rem}
The operator \eqref{Cdef0} with absorbing boundary conditions \eqref{DcDDbc}, or absorbing on the left and reflecting on the right \eqref{PScaseAR}, is not dissipative \cite{Sankaranarayanan2014}.  Hence, by the Lumer-Phillips Theorem \cite[Theorem 3.4.5]{ABHN2}, it cannot generate a contraction semigroup.  In our setting, a contraction semigroup means that
\[\int_0^1 |u(x,t)|\,dx\leq \int_0^1 |u_0(x)|\,dx ,\]
i.e., if the solution stays positive, then no mass can be created.  Since any physically meaningful diffusion equation must satisfy this condition, we see another reason why the Caputo fractional differential equation \eqref{CaputoFDE} is not a suitable to model anomalous diffusion.
\end{rem}

\begin{rem}
If one considers the Caputo fractional differential equation \eqref{CaputoFDE} on $0\leq x\leq 1$ with $1<\alpha<2$ and the traditional Neumann boundary conditions
\begin{equation}
\label{CaputoNofluxBC}
\frac{\partial}{\partial x}u(0,t)=\frac{\partial}{\partial x} u(1,t)=0\quad\text{for all $t\geq 0$,}
\end{equation}
then the Caputo and Patie-Simon fractional derivatives are equal, in light of \eqref{PStoCaputo}.  Since \eqref{CaputoNofluxBC} implies \eqref{PSnofluxBC} by \eqref{Cdef0}, solutions to \eqref{CaputoFDE} with the boundary conditions \eqref{CaputoNofluxBC} also solve the problem \eqref{C-CauchyEq0} with reflecting boundary conditions \eqref{PSnofluxBC}.   However, the domain of the fractional derivative \eqref{PSdvt} with the reflecting boundary conditions \eqref{PSnofluxBC} is strictly larger, and there are solutions to \eqref{C-CauchyEq0} with reflecting boundary conditions \eqref{PSnofluxBC} that do not solve \eqref{CaputoFDE} with the boundary conditions \eqref{CaputoNofluxBC}, e.g., note that $\frac{\partial}{\partial x}u(1,t)\neq 0$ in  Figure \ref{DcNNfig}.
\end{rem}


\section*{Appendix}
The following result implies that solutions in this paper converge to the steady state solution $u=0$ if at least one boundary condition is absorbing. This follows because, in this case, the fractional derivative operator is invertible and generates a strongly continuous positive contraction semigroup \cite{Sankaranarayanan2014}.

\begin{lemma}
Let $(\Omega,\mu)$ be a $\sigma$-finite measure space and $X=L_p(\Omega)$, $1\le p <\infty$, or let $\Omega$ be a locally compact Hausdorff space and $X=C_0(\Omega)$. Suppose that $A$ generates a strongly continuous positive contraction semigroup on $X$ and that $A^{-1}$ exist as a bounded operator on $X$. Then, for all $x\in X$, we have $\|T(t)x\|\to 0$ as $t\to \infty$ exponentially fast.
 \end{lemma}
 \begin{proof}
 Let $\sigma(A)$ denote the spectrum of $A$ and $\rho(A)$ the resolvent set of $A$. Since $A$ generates a strongly continuous contraction semigroup, it follows from the Hille-Yosida Theorem that $(0,\infty)\subset \rho(A)$. As $A^{-1}$ exist as a bounded operator on $X$, it follows that $0\in \rho(A)$. Since $A$ generates a strongly continuous positive contraction semigroup on $X$, it follows that the resolvent $R(\lambda, A)$ of $A$ satisfies $R(\lambda, A)=\int_0^{\infty}e^{-\lambda t}T(t)\,\, dt\ge 0$ (strong Bochner integral) for $\lambda>0$ as the positive cone is closed in $X$. The resolvent is an analytic function of $\lambda$ for $\lambda \in \rho(A)$ and hence continuous. Therefore $A^{-1}=\lim_{\lambda \to 0+}R(\lambda, A)\ge 0$, again, since the positive cone is closed in $X$.

 Let $s(A)$ denote the spectral bound of $A$; that is, $$s(A)=\sup\{\text{Re}\,\lambda :\,\lambda \in \sigma(A)\}$$ and $\omega_0\in \mathbb{R}$ the growth bound of the semigroup; that is, $$\omega_0=\inf\{\omega\in \mathbb{R}:\text{ there is }M_{\omega}\ge 1\text{ such that }\|T(t)\|\leq M_{\omega}e^{\omega t} ,\quad t\ge 0\}.$$ It follows from \cite[Chapter VI, Lemma 1.9]{En-Na} that $A^{-1}\ge 0$ implies that $s(A)<0$.  Finally, by \cite[Theorem 5.3.6]{ABHN2} when $X=L_p(\Omega)$ and \cite[Theorem 5.3.8]{ABHN2} when $X=C_0(\Omega)$, it follows that $\omega_0=s(A)$ and hence $\omega_0<0$. Thus, there is $\epsilon>0$ and $M_{\epsilon}\ge 1$  such that $\|T(t)\|\leq M_{\epsilon}e^{-\epsilon t}$, $t\ge 0$, and the proof is complete.

%
%
\end{proof}


\begin{thebibliography}{10}


\bibitem{ABHN2}
W.~Arendt, C.~J.~K. Batty, M.~Hieber, and F.~Neubrander,
\emph{Vector-valued {L}aplace transforms and {C}auchy problems}, 2nd Ed., Monographs in Mathematics {\bf 96}, Birkh\"auser Verlag, Basel, 2010.

\bibitem{BKStams} B.~Baeumer, M.~Kov\'acs, and H.~Sankaranarayanan, Higher order Grunwald approximations of fractional derivatives and fractional powers of operators. {\it Trans. Amer. Math. Soc.} {\bf 367} (2015), 813--834.
	
\bibitem{ReflectedStable}
B.~Baeumer, M.~Kov\'acs, M.~M.~Meerschaert, R.~L.~Schilling, and P.~Straka, Reflected spectrally negative stable processes and their governing equations, \emph{Trans. Amer. Math. Soc.}, {\bf 368} No. 1 (2016), 227--248.

\bibitem{SpaceTimeFrac}
B.~Baeumer, T.~Luks, and M.~M.~Meerschaert, Space-time fractional Dirichlet problems.  Preprint at {\tt www.stt.msu.edu/users/mcubed/SpaceTimeFrac.pdf}.

\bibitem{CushmanGinn}
J.~H. Cushman and T.~Ginn,  Fractional advection-dispersion equation: A classical mass balance with convolution-Fickian flux, \emph{Water Resour. Res.} \textbf{36} (2000),
  3763--3766.

\bibitem{FCAAnonlocal}
O.~Defterli, M.~D'Elia, Q.~Du, M.~Gunzburger, R.~Lehoucq, and M.~M.~Meerschaert, Fractional diffusion on bounded domains, \emph{Fract. Calc. Appl. Anal.}, {\bf 18} (2015), No. 2, pp. 342--360.

\bibitem{Deng}
Z.~Deng, V.~P.~Singh, and L.~Bengtsson, Numerical solution of fractional advection-dispersion equation, \emph{J. Hydraul. Eng.}, {\bf 130}  (2004), 422--431.

\bibitem{DFF}
K.~Diethelm, N.~J.~Ford, and A.~D.~Freed, A predictor-corrector approach for the numerical solution of fractional differential equations, \emph{Nonlinear Dynamics}, {\bf 29}, No. 1--4 (2002), 3--22.

\bibitem{En-Na}K.~J.~Engel and R.~Nagel, {\it One-Parameter
Semigroups for Linear Evolution Equations}, Springer-Verlag, 2000.
	
\bibitem{FixRoop}
G.~J.~Fix and J.~P.~Roop, Least squares finite element solution of a fractional order two-point boundary value problem, \emph{Comput. Math. Appl.},  {\bf 48} (2004), 1017--1033.
	
\bibitem{Herrmann}
R.~Herrmann, \emph{Fractional Calculus: An Introduction for Physicists}, 2nd Ed., World Scientific, Singapore (2014).

\bibitem{KRS}
R.~Klages, G.~Radons, and I.~M.~Sokolov, \emph{Anomalous Transport: Foundations and Applications}, Wiley-VCH, Weinheim, Germany (2008).

\bibitem{LAT}
 F.~Liu, V.~Ahn, and I.~Turner, Numerical solution of the space fractional Fokker-Planck equation, \emph{J. Comput. Appl. Math.}, {\bf 166}  (2004),  209--219.

\bibitem{Lynch}
V.~E.~Lynch, B.~A.~Carreras, D.~del-Castillo-Negrete, K.~M.~Ferreira-Mejias, and H.~R.~Hicks, Numerical methods for the solution of partial differential
equations of fractional order, \emph{J. Comput. Phys.},  {\bf 192}  (2003), 406--421.


\bibitem{MainardiWaves}
F.~Mainardi, \emph{Fractional Calculus and Waves in Linear Viscoelasticity: An Introduction to Mathematical Models}, World Scientific  (2010).

\bibitem{mainardi1997fractals}
F.~Mainardi, \emph{Fractals and Fractional Calculus in Continuum Mechanics}, Springer, Berlin (1997).

\bibitem{frade}
M.~M.~Meerschaert and C.~Tadjeran, Finite difference approximations for fractional advection-dispersion flow equations, \emph{J. Comput. Appl. Math.},  {\bf 172} (2004), 65--77.

\bibitem{2sided}
M.~M.~Meerschaert and C.~Tadjeran, Finite difference approximations for two-sided space-fractional partial differential equations, \emph{Appl. Numer. Math.}, {\bf 56} (2006), 80--90.

\bibitem{TMSADIIE}
M.~M.~Meerschaert, H.~P.~Scheffler, and C.~Tadjeran, Finite difference methods for two-dimensional fractional dispersion equation, \emph{J. Comput. Phys.}, {\bf 211}  (2006), 249--261.

\bibitem{FCbook}
M.~M.~Meerschaert and A.~Sikorskii, \emph{Stochastic Models for Fractional Calculus}, De Gruyter Studies in Mathematics {\bf 43}, De Gruyter, Berlin (2012).

\bibitem{MathModeling}
M.~M.~Meerschaert, \emph{Mathematical Modeling}, 4th Ed., Academic Press (2013).

\bibitem{Metzler2000}
R.~Metzler and J.~Klafter, The random walk's guide to anomalous diffusion: a fractional dynamics approach, \emph{Phys. Rep.}, {\bf 339}(1) (2000), 1--77.

\bibitem{Metzler2004a}
R.~Metzler and J.~Klafter, The restaurant at the end of the random walk: recent developments in the description of anomalous transport by fractional dynamics, \emph{J. Phys. A}, {\bf 37}(31) (2004), R161--R208.

\bibitem{OM}
Z.~Odibat and S.~Momani, Numerical methods for nonlinear partial differential equations of fractional order, \emph{Appl. Math. Model.}, {\bf 32}(1)  (2008), 28--39.

\bibitem{Oldham1974}
K.~B.~Oldham and J.~Spanier,  \emph{The Fractional
  Calculus Theory and Applications of Differentiation and Integration to arbitrary order}, Mathematics in Science and Engineering  {\bf 111}, Elsevier Science (1974).

 \bibitem{PatieSimon}
 P.~Patie and T.~Simon, Intertwining certain fractional derivatives,  \emph{Potential Anal.}, \textbf{36} (2012), 569--587.

\bibitem{Podlubny}
I.~Podlubny, \emph{Fractional differential equations: An introduction to fractional derivatives, fractional differential equations, to methods of their solution and some of their applications}, Academic Press, San Diego, California (1999).

\bibitem{Richtmyer}
R.~D.~Richtmyer and K.~W.~Morton, \emph{Difference methods for initial-value problems}, Krieger Publishing, Malabar, Florida (1994).

\bibitem{Samko}
S.~Samko, A.~Kilbas, and O.~Marichev, \emph{Fractional Integrals and derivatives: Theory and Applications}, Gordon and Breach, London (1993).

\bibitem{Sankaranarayanan2014}
H.~Sankaranarayanan, \emph{{G}r\"unwald-type approximations and boundary conditions for one-sided fractional derivative operators}, Ph.D. thesis, University of Otago, New Zealand (2014).  Available at {\tt hdl.handle.net/10523/5216}

\bibitem{Schumer2001}
R.~Schumer, D.~A. Benson, M.~M.~Meerschaert, and S.~W.~Wheatcraft, Eulerian derivation for the fractional advection-dispersion equation, \emph{J. Contam. Hydrol.}, \textbf{48} (2001), 69--88.

\bibitem{CNfde}
C.~Tadjeran, M.~M.~Meerschaert, and H.~P.~Scheffler, A second order accurate numerical approximation for the fractional diffusion equation, \emph{J. Comput. Phys.}, {\bf 213} (2006), 205--213.

\bibitem{Yuste}
S.~B.~Yuste and L.~Acedo, An explicit finite difference method and a new von Neumann type stability analysis for fractional diffusion equations, \emph{SIAM J. Numer. Anal.}, {\bf 42}  (2005), 1862--1874.

\bibitem{ZLPM}
H.~Zhang, F.~Liu, M.~S.~Phanikumar, and M.~M.~Meerschaert, A novel numerical method for the time variable fractional order mobile-immobile advection-dispersion model, \emph{Comput. Math. Appl.}, {\bf 66} No. 5 (2013), 693--701.


\end{thebibliography}
\end{document}